\documentclass[10pt]{amsart}

\usepackage{hyperref}
\usepackage{amsmath}
\usepackage{graphicx}
\usepackage{tikz}
\usepackage{tikz-cd}
\usepackage{tikz-3dplot}
\usepackage{subcaption}

\usetikzlibrary{snakes}
\usepackage[all]{xy}
\usepackage{framed}
\usepackage{mathrsfs}
\usepackage{amssymb}
\usepackage{stmaryrd}

\usepackage{mathabx}

\usepackage{colonequals} 
\usepackage{enumitem} 

\usepackage{textcomp} 

\usepackage{mathtools}

\usepackage{amstext} 
\usepackage{array}   
\newcolumntype{C}{>{$}c<{$}} 

\usetikzlibrary{shapes.geometric, calc}

\theoremstyle{plain}
\newtheorem{theorem}{Theorem}[section]
\newtheorem{lemma}[theorem]{Lemma}
\newtheorem{proposition}[theorem]{Proposition}

\newtheorem{question}[theorem]{Question}

\theoremstyle{definition}

\newtheorem{remark}[theorem]{Remark}

\numberwithin{equation}{section}

\newcommand{\QQ}{\mathbb{Q}}
\newcommand{\RR}{\mathbb{R}}
\newcommand{\ZZ}{\mathbb{Z}}

\newcommand{\im}{{\rm im}}

\DeclareMathOperator{\Hom}{Hom}

\makeatletter
\@namedef{subjclassname@2020}{\textup{2020} Mathematics Subject Classification}
\makeatother

\begin{document}

\title[]{Toric surfaces with symmetries by reflections}

%
%

\author{Jongbaek Song}
\address{School of Mathematics, KIAS, 
85 Hoegiro Dongdaemun-gu, Seoul 02455, Republic of Korea}
\email{jongbaek@kias.re.kr}

%
%
%
%



\maketitle 
\abstract 
Let $W$ be a reflection group in a plane and $P$ a rational polygon that is invariant under the $W$-action. The action of $W$ on $P$ induces a $W$-action on the toric variety $X_P$ associated with $P$. In this paper, we study the $W$-representation on the cohomology $H^\ast(X_P)$ and show that the invariant subring $H^\ast(X_P)^W$ is isomorphic to the cohomology ring of the toric variety associated with the \emph{fundamental region}~$P/W$. 
As an example, we provide an explicit description of the main result for the case of the toric variety associated with the fan of Weyl chambers of type $G_2$. 
\endabstract

\section{Introduction}
A toric variety is a complex algebraic variety on which an algebraic torus acts with an open dense orbit. Due to the nice torus symmetry of a toric variety, 
one can construct a toric variety~$X$ from a certain combinatorial object called a \emph{fan} $\Sigma$, and  the combinatorics of $\Sigma$ can be recovered using the orbit structure of the torus action on $X$. For the case of a projective toric variety, we often begin with a lattice $M$ and a rational polytope $P\subset M\otimes_\ZZ \RR$ to define a toric variety $X_P$ by considering the normal fan $\Sigma_P$ of $P$ (see \cite[Section  2.3]{CLS} for instance). In general, two different polytopes may define the same normal fan. In this case, they define the same toric variety with different projective embeddings. 

Recall that the normal fan of $P$ is determined by the face structure of $P$ together with primitive vectors~$\lambda_i\in M^\ast$ defining the hyperplane supporting 
each codimension-1 face  $F_i$ of $P$. Therefore, most of the topological information of $X_P$ is encoded in the face structure of $P$ and~$\lambda_i$'s. For instance, when $P$ is a simple polytope, the associated toric variety $X_P$ is an orbifold. In this case, the cohomology ring $H^\ast(X_P;\QQ)$ is given by 
\begin{equation}\label{eq_cohom_XP}
H^\ast(X_P;\QQ)=\QQ[x_1, \dots, x_m]/\mathcal{I}+\mathcal{J},
\end{equation}
where $\mathcal{I}$ and $\mathcal{J}$ are ideals 
\begin{enumerate}
\item $\mathcal{I}=\left<\prod_{k=1}^r x_{i_k} \mid \bigcap_{k=1}^r  F_{i_k} =\emptyset \right>$; 
\item $\mathcal{J}=\left< \sum_{k=1}^m \left< m, \lambda_i\right>x_i \mid m\in M\right> $. 
\end{enumerate}
This result is due to Jurkiewicz \cite{Jur3} for projective nonsingular toric varieties and to Danilov \cite{Dan} for arbitrary toric orbifolds.

When a toric variety $X_P$ arises from a polytope $P\subset \RR^n$ that is preserved by an action of a finite group $W\varleq GL_n(\RR)$, the action of $W$ on $P$ determines a $W$-action on $X_P$. Hence, the cohomology $H^\ast(X_P)$ is equipped with the induced $W$-module structure. A typical example is given by the toric variety associated with the permutohedron~$P_n$, namely the convex hull of the orbit of $(1, 2, \dots, n+1)\in \RR^{n+1}$ by the action of the symmetric group $\mathfrak{S}_{n+1}$ defined by permuting coordinates.  Instead of $\mathfrak{S}_{n+1}$ and $(1, \dots, n+1)\in \RR^{n+1}$, one can consider the Weyl group for each of other Lie types and a point having distinct coordinates in the corresponding root lattice respectively, which yields a lattice polytope known as a \emph{weight polytope} in the literature.  Accordingly, we have the associated  toric variety. 
We refer the readers to \cite{Proc}, where the author considered the fan consisting of Weyl chambers, that is the normal fan of the permutohedron. We also refer to \cite{Abe, BaBl, Huh, Kly, MPS, DoLu, Stem} for relevant studies of these varieties. 

In this paper, we consider arbitrary 2-dimensional polytopes $P$ that are preserved by actions of groups $W$ generated by reflections. Note that the reflection group of a plane is the dihedral group 
\begin{equation}\label{eq_dihedral_gp}
D_{2\ell}=\left< s_1, s_2 \mid s_1^2=s_2^2=1,~(s_1s_2)^\ell=(s_2s_1)^\ell \right>
\end{equation}
generated by two reflections $s_1$ and $s_2$ corresponding to two lines with angle $\frac{2\pi}{\ell}$ for some $\ell\geq 3$. Therefore, we consider the cases where the group $W$ is either the group generated by a single reflection or the dihedral group $D_{2\ell}$. In either cases, one can take a region $R\subset P$ such that each point $\mathbf{x}\in R$ represents the $W$-orbit of~$\mathbf{x}$. We denote by $P/W$ the closure $\overline{R}\subset P$ and call it the \emph{fundamental region}. Notice that~$P/W$ is again a convex rational polytope, hence we may associate $P/W$ with a toric variety $X_{P/W}$. Then, we study the $W$-action on the cohomology $H^\ast(X_P)$ and show the following theorem. 
\begin{theorem}\label{thm_main}
Let $P$ be a 2-dimensional rational polytope with an action of a reflection group $W$. Let $X_P$ and $X_{P/W}$ be toric varieties associated with $P$ and the fundamental region $P/W$, respectively. Then, there is a ring isomorphism 
\begin{equation*}
H^\ast(X_{P/W}) \cong H^\ast(X_P)^W.
\end{equation*}
\end{theorem}
Using the following classical result  
\[
H^\ast(X/W) \cong H^\ast(X)^W
\]
for any locally compact Hausdorff $W$-space for some finite group $W$ (see for instance \cite[III-2]{Bor}), the isomorphism of Theorem \ref{thm_main} also implies 
\begin{equation}\label{eq_claim}
H^\ast(X_{P/W}) \cong H^\ast(X_P/W).
\end{equation}

One of the motivations of Theorem \ref{thm_main} is the following question posed in \cite{HMSS}. 
\begin{question}\cite[Question 8.1]{HMSS} \label{question}
If $W$ is generated by reflections in $GL_n(\RR)$ and acts on $P$, must $X_P/W$ be isomorphic with $X_{P/W}$?
\end{question}
Several classes of toric varieties supporting Question~\ref{question} are provided in  \cite[Section 8]{HMSS}. The result of Theorem \ref{thm_main} gives a cohomological evidence of this question for toric surfaces.

This paper is organized as follows. 
We begin in Section \ref{sec_toric_surface} with a brief background about toric surfaces, where we mostly focus on their rational cohomology rings. For simplicity, we write $H^\ast(-)\colonequals H^\ast(-;\QQ)$ otherwise stated. 

In Sections \ref{sec_sing_ref} and \ref{sec_two_ref},  we consider the case where $P$ is symmetric by a single reflection and by a dihedral group respectively, and give a proof of Theorem \ref{thm_main} for each case. Although the main ideas are similar, we separate them to give a more detailed explanation for each case.  

Section \ref{sec_G_2} is devoted to the toric variety associated with the fan of Weyl chambers of type~$G_2$ as a particular case of Theorem \ref{thm_main}. Instead of working with a fan, we consider a polytope whose normal fan is the fan of Weyl chambers of type~ $G_2$, see \eqref{eq_fund_reg_G_2}. Indeed, the authors of \cite{HMSS} considered toric varieties associated with fans of Weyl chambers of all classical Lie types and proved the isomorphism of Theorem~\ref{thm_main} for any parabolic subgroup of the Weyl group, see \cite[Theorem 1.1]{HMSS}. Section~\ref{sec_G_2} extends their result for the root system of type $G_2$. 

\section{Toric surface}\label{sec_toric_surface}
Let $T(\cong (S^1)^n)$ be an $n$-dimensional torus. We denote by $M$ and $N$  the character lattice of $T$ and the lattice of 1-parameter subgroups, respectively. Given a character $\chi^m\in \Hom(S^1, T)$ for $m \in M$ and a $1$-parameter subgroup $\lambda^n\in \Hom (T, S^1)$ for $n\in N$, we have $\chi^m \circ \lambda^n \colon S^1 \to S^1$ given by $t \mapsto t^r$ for some $r\in \ZZ$. Hence, there is a natural bilinear pairing 
\[
\left< ~, ~\right> \colon M \times N \to \ZZ
\]
defined by $\left<m, n\right>=r$. 

Let $P$ be an $n$-dimensional rational polytope in $M\otimes_\ZZ \RR$ with $\ell$-facets, say $F_1, \dots, F_\ell$.  We denote by $\lambda_i\in N$ the primitive (outward) vector defining a facet $F_i$ for $1 \leq i \leq \ell$, namely, 
\begin{equation}\label{eq_P_as_intersection_of_half_sp}
P= \bigcap_{i=1}^\ell \left\{\mathbf{x}\in M\otimes_\ZZ \RR \mid \left< \mathbf{x}, \lambda_i\right> + a_i \leq 0 \text{ for } \lambda_i\in N,~ a_i<0\right\}
\end{equation}
and $F_i=\{\mathbf{x}\in P \mid  \left< \mathbf{x}, \lambda_i\right> + a_i = 0\}$. We assume that no redundant inequality exists in \eqref{eq_P_as_intersection_of_half_sp} without loss of generality. 
Then, one can associate a toric variety $X_P$ of complex dimension $n$. We refer to \cite{CLS} for explicit definitions of $X_P$. Also, we refer to \cite[Chapter 7]{BP-book} for more topological viewpoint.  When $X_P$ is of complex dimension $2$, namely when it is associated to a rational polygon, we call $X_P$ a toric surface. 

In this manuscript, we focus on the rational cohomology ring $H^\ast(X_P)$ of a toric surface $X_P$. So, the ideal $\mathcal{I}$ of \eqref{eq_cohom_XP} can be written by 
\begin{equation}\label{eq_ideal_polygon}
\mathcal{I}=\left< x_ix_j \mid E_i\cap E_j =\emptyset \right>,
\end{equation}
where $E_i$ denotes a facet, namely an edge of $P$. 
%
%

We notice that $H^\ast(X_P)$ is a Poincar\'e duality algebra as $X_P$ is an orbifold, see for instance \cite[Theorem 11.4.8]{CLS}. The following proposition about Poincar\'e duality algebras will be used in the proof of Theorem \ref{thm_main}, which will be discussed in the following two sections. 

\begin{proposition}\label{prop_PD}
Let $\phi =\bigoplus_{i=1}^n \phi_n \colon \mathcal{A}=\bigoplus_{i=0}^nA_i \to \mathcal{B}=\bigoplus_{i=0}^nB_i$ be a morphism of graded algebras. If $\mathcal{A}$ is a Poincar\'e duality algebra and $\phi_{n} \colon A_n \to B_n$ is an isomorphism, then $\phi$ is an isomorphism. 
\end{proposition}

\section{Symmetry by a reflection}\label{sec_sing_ref}
In this section, we  consider a polygon $P\subset M\otimes_\ZZ \RR$ having a symmetry by a reflection $\sigma$, namely $\sigma \colon M\otimes_\ZZ \RR \to M\otimes_\ZZ \RR$ sending $P$ onto itself. Then, $\sigma$-action on $P$ induces the $\sigma$-action on the corresponding toric surface $X_P$. Therefore, we have the $\sigma$-action on the cohomology $H^\ast(X_P)$ defined by $s(x_i)=x_j$ if $\sigma(F_i)=F_j$. 

The fixed point set of $P$ by $\sigma$ is 
\begin{equation}\label{eq_new_edge}
E_\sigma\colonequals \{ \mathbf{x}\in P \mid \sigma(x)=x\}=\{ \mathbf{x}\in P \mid \left<\mathbf{x}, \eta\right>=0\},
\end{equation}
for some primitive vector $\eta\in N$. We take the half space 
\begin{equation*}
P/\sigma\colonequals \{ \mathbf{x}\in P \mid \left<\mathbf{x}, \eta\right>\leq 0\}
\end{equation*} 
and call it the \emph{fundamental region} of $P$ with respect to $\sigma$. We note that $P/\sigma$ itself is a polygon with edges inherited from $P$ together with the one extra edge $E_\sigma$. 

\begin{figure}
\begin{tikzpicture}
\node[opacity=0, regular polygon, regular polygon sides=12, draw, minimum size = 2.5cm](m) at (0,0) {};
\foreach \x in {1,...,12}{
\coordinate (\x) at (m.corner \x); 
}

\draw[fill=yellow, yellow] (m.side 7)--(8)--(9)--(m.side 9) [out=65, in=-65] to (m.side 11)--(12)--(1)--(m.side 1);

\node[right] at (0,0) {\scriptsize$E_\sigma$};

\foreach \a in {1,2,3,12,6,7,8,9} {
\draw[fill] (\a) circle (1.5pt); 
}

\draw (m.side 11)--(12)--(1)--(2)--(3)--(m.side 3);
\draw (m.side 5)--(6)--(7)--(8)--(9)--(m.side 9);

\node[rotate=90] at (m.side 4) {$\cdots$};
\node[rotate=90] at (m.side 10) {$\cdots$};

\node at ($1.2*(m.side 12)$) {\scriptsize$E_1$};
\node[rotate=-60] at ($1.2*(m.side 11)$) {\scriptsize$\cdots$};

\node at ($1.2*(m.side 8)$) {\scriptsize$E_n$};
\node[rotate=60] at ($1.2*(m.side 9)$) {\scriptsize$\cdots$};

\node at ($1.2*(m.side 2)$) {\scriptsize$E_{n+1}$};
\node[rotate=60] at ($1.2*(m.side 3)$) {\scriptsize$\cdots$};

\node at ($1.2*(m.side 6)$) {\scriptsize$E_{2n}$};
\node[rotate=120] at ($1.2*(m.side 5)$) {\scriptsize$\cdots$};

\node at ($1.2*(m.side 7)$) {\scriptsize$E_{2n+1}$};
\node at ($1.2*(m.side 1)$) {\scriptsize$E_{2n+2}$};

\draw[blue] ($0.8*(1)+0.8*(2)$)--($0.7*(7)+0.7*(8)$);
\node[above] at ($0.8*(1)+0.8*(2)$) {$\sigma$};
\node at ($0.75*(1)+0.75*(2)$) {$\curvearrowleftright$};

\node at (0,-2.2) {Case (1-1)};

\begin{scope}[xshift=120]
\node[opacity=0, regular polygon, regular polygon sides=9, draw, minimum size = 2.5cm](m) at (0,0) {};
\foreach \x in {1,...,9}{
\coordinate (\x) at (m.corner \x); 
}

\draw[fill=yellow, yellow] (m.side 5)--(6)--(7)--(m.side 7) [out=90, in=-70] to (m.side 8)--(9)--(1);
\node[right] at (0,0) {\scriptsize$E_\sigma$};

\foreach \x in {1,2,4,5,6,7,9}{
\draw[fill] (\x) circle (1.5pt); 
}

\draw (m.side 8)--(9)--(1)--(2)--(m.side 2);
\draw (m.side 3)--(4)--(5)--(6)--(7)--(m.side 7);

\node[rotate=80] at ($0.9*(3)$) {$\cdots$};
\node[rotate=100] at ($0.9*(8)$) {$\cdots$};

\draw[blue] ($1.5*(1)$)--($0.7*(5)+0.7*(6)$);
\node[above] at ($1.5*(1)$) {$\sigma$};
\node at ($1.4*(1)$) {$\curvearrowleftright$};

\node at ($1.2*(m.side 9)$) {\scriptsize$E_{1}$};
\node[rotate=110] at ($1.2*(m.side 8)$) {\scriptsize$\cdots$};

\node at ($1.2*(m.side 6)$) {\scriptsize$E_{n}$};
\node[rotate=80] at ($1.2*(m.side 7)$) {\scriptsize$\cdots$};

\node at ($1.2*(m.side 1)$) {\scriptsize$E_{n+1}$};
\node[rotate=70] at ($1.2*(m.side 2)$) {\scriptsize$\cdots$};

\node at ($1.2*(m.side 4)$) {\scriptsize$E_{2n}$};
\node[rotate=100] at ($1.2*(m.side 3)$) {\scriptsize$\cdots$};

\node at ($1.2*(m.side 5)$) {\scriptsize$E_{2n+1}$};

\node at (0,-2.2) {Case (1-2)};

\end{scope}

\begin{scope}[xshift=240]
\node[rotate=22.5, opacity=0, regular polygon, regular polygon sides=8, draw, minimum size = 2.5cm](m) at (0,0) {};
\foreach \x in {1,...,8}{
\coordinate (\x) at (m.corner \x); 
}

\draw[fill=yellow, yellow] (5)--(6)--(m.side 6) [out=80, in=-80] to (m.side 7)--(8)--(1);
\node[right] at (0,0) {\scriptsize$E_\sigma$};

\foreach \x in {8,1,2,4,5,6}
{\draw[fill] (\x) circle (1.5pt); }

\draw (m.side 7)--(8)--(1)--(2)--(m.side 2);
\draw (m.side 3)--(4)--(5)--(6)--(m.side 6);

\node[rotate=90] at ($0.9*(3)$) {$\cdots$};
\node[rotate=90] at ($0.9*(7)$) {$\cdots$};

\node at ($1.2*(m.side 1)$) {\scriptsize$E_{n+1}$};
\node[rotate=67.5] at ($1.2*(m.side 2)$) {\scriptsize$\cdots$};

\node at ($1.2*(m.side 4)$) {\scriptsize$E_{2n}$};
\node[rotate=112.5] at ($1.2*(m.side 3)$) {\scriptsize$\cdots$};

\node at ($1.2*(m.side 5)$) {\scriptsize$E_{n}$};
\node[rotate=67.5] at ($1.2*(m.side 6)$) {\scriptsize$\cdots$};

\node at ($1.2*(m.side 8)$) {\scriptsize$E_{1}$};
\node[rotate=-67.5] at ($1.2*(m.side 7)$) {\scriptsize$\cdots$};

\draw[blue] ($1.5*(1)$)--($1.3*(5)$);
\node[above] at ($1.5*(1)$) {$\sigma$};
\node at ($1.4*(1)$) {$\curvearrowleftright$};
\node at (0,-2.2) {Case (1-3)};

\end{scope}
\end{tikzpicture}
\caption{Polygons with symmetries by a reflection and their fundamental regions.}
\label{fig_symmetry_by_a_reflection}
\end{figure}
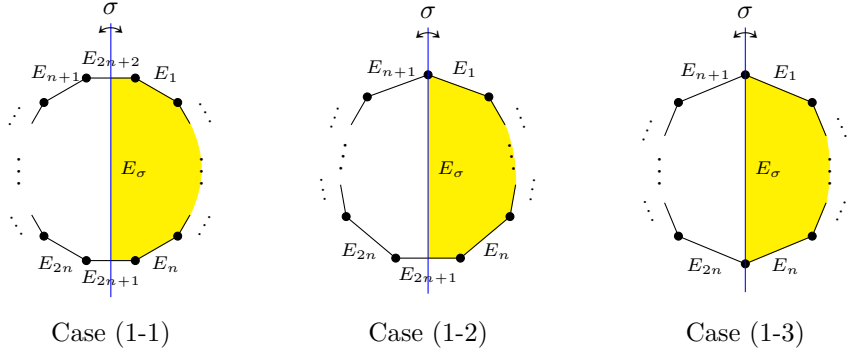

For the convenience of the later discussion, we consider three different cases of the symmetry of $P$ and the resulting fundamental regions $P/\sigma$ described in Figure~ \ref{fig_symmetry_by_a_reflection}. In the following subsections, we prove  Theorem \ref{thm_main} for each Case (1-1)--(1-3) together with the group $W$ generated by~$\sigma$, where we use the labelings of edges  given in Figure~ \ref{fig_symmetry_by_a_reflection}. 

\subsection{Proof of Theorem \ref{thm_main} for Case (1-1)}\label{subsec_Case_1_1}
Recall from \eqref{eq_P_as_intersection_of_half_sp} and \eqref{eq_new_edge} that $\lambda_i\in N$ and $\eta\in N$ are vectors defining the edge $E_i$ and $E_\sigma$, respectively. Since $P$ is symmetric by the reflection $\sigma$,  $\lambda_i$'s and $\eta$ satisfy the following three conditions:
\begin{enumerate}[label=(\roman*)]
\item  \label{eq_ci}  $\lambda_{n+i}-\lambda_i = c_i \eta$ for some $c_i\in \QQ$ and for each $i =1, \dots,  n$; 
\item \label{eq_2n+1_perp_eta} $\lambda_{2n+1} \perp \eta$;
\item \label{eq_2n+2_parallel_2n+1} $\lambda_{2n+1}=-\lambda_{2n+2}$.
\end{enumerate}

Applying  \eqref{eq_cohom_XP}, we have 
\[
H^\ast(X_{P/\sigma})=\QQ[x_1, \dots, x_{n}, x_{2n+1}, x_{2n+2}, x_\sigma]/\mathcal{I}_{\text{(1-1)}}+\mathcal{J}_{\text{(1-1)}}
\]
where $\mathcal{I}_{\text{(1-1)}}$ is the ideal generated by
\[
\left\{ x_jx_k\mid  E_j \cap E_k = \emptyset  \text{ for } j,k \in \{1, \dots, n, 2n+1, 2n+2\} \right\}\cup \{x_jx_\sigma\mid 1\leq j\leq n\}
\]
and $\mathcal{J}_{\text{(1-1)}}$ is the ideal generated by 
\[
\left\{ \sum_{i=1}^n \left< m, \lambda_i\right> x_i + \left< m, \lambda_{2n+1}\right>x_{2n+1}+ \left< m, \lambda_{2n+2}\right>x_{2n+2} + \left<m, \eta\right> x_\sigma ~\Big|~  m\in M\right\}.
\]

Now, we consider the map 
\begin{equation}\label{eq_map_phi}
\phi \colon H^\ast(X_{P/\sigma}) \to H^\ast(X_P)
\end{equation}
defined by 
\begin{align*}
x_i & \mapsto \begin{cases} x_i + x_{n+i} &  i=1, \dots, n; \\
x_i & i=2n+1,~ 2n+2,\end{cases} \\
x_\sigma &\mapsto \sum_{i=1}^nc_ix_{n+i},
\end{align*}
where $c_i\in \ZZ$ is defined in \ref{eq_ci}. 

\begin{lemma}\label{lem_image_invariant}
The image of $\phi$ is the invariant ring $H^\ast(X_P)^\sigma$. 
\end{lemma}
\begin{proof}
Obviously, $\phi(x_i)$ is $\sigma$-invariant for each  $i \in \{1, \dots, n, 2n+1, 2n+2\}$. Below, we prove that $\phi(x_\sigma)$ is $\sigma$-invariant. Since $\lambda_{2n+1}\perp \eta$, we take $\{\eta_1\colonequals \eta,\eta_2\colonequals \lambda_{2n+1} \}$ as a basis of $N$. Let $\{\eta_1^\ast, \eta_2^\ast \}\subset M$ be the dual basis of $\{\eta_1, \eta_2\}$, i.e., $\left<\eta_{i}^\ast, \eta_j\right>=\delta_{ij}$ for $1\leq i, j \leq 2$. 

Recall from \eqref{eq_cohom_XP} that 
\[
\sum_{i=1}^{2n+2}\left<m, \lambda_i \right> x_i =0\in H^\ast(X_P).
\]
Applying $m=\eta^\ast$ and  $\lambda_{n+i}=\lambda_{i}+c_i\eta$ for $i=1, \dots, n$, we have 
\begin{equation}\label{eq_linear_single_ref}
\sum_{i=1}^n\left<\eta^\ast, \lambda_{i}\right> x_i + \sum_{i=1}^n\left<\eta^\ast, \lambda_{i}\right> x_{n+i} +\sum_{i=1}^n c_ix_{n+i} 
=\sum_{i=1}^n\left<\eta^\ast, \lambda_{i}\right> (x_i+x_{n+i})+\phi(x_\sigma)
\end{equation}
which is $0$ in $H^\ast(X_P)$. Since the first term of the right-hand side of \eqref{eq_linear_single_ref} is $\sigma$-invariant, we conclude that so is $\phi(x_\sigma)$. Therefore, we have $\im \phi \subset H^\ast(X_P)^\sigma$. 

To show $H^\ast(X_P)^\sigma \subset \im \phi$, recall that the torus equivariant cohomology of $X_P$ is 
$H^\ast_T(X_P)=\QQ[x_1, \dots, x_{2n+2}]/\mathcal{I}$, where $\mathcal{I}$ is the ideal of \eqref{eq_ideal_polygon} (for instance, see \cite[Section 12.4]{CLS}). It is straightforward to see that the invariant ring $H^\ast_T(X_P)^\sigma$ is generated by $\{x_i+x_{n+i}\mid i=1, \dots, n\}\cup \{x_{2n+1}, x_{2n+2}\}$. Since the natural surjective homomorphism $H^\ast_T(X_P)\to H^\ast(X_P)$ is $\sigma$-equivariant, its restriction $H^\ast_T(X_P)^\sigma \to H^\ast(X_P)^\sigma$ is also surjective. 
\end{proof}

\begin{lemma}\label{lem_well_def}
The map $\phi$ is well-defined. 
\end{lemma}
\begin{proof}
It suffices to show that $\ker \phi$ contains both $\mathcal{I}_{\text{(1-1)}}$ and $\mathcal{J}_{\text{(1-1)}}$. Indeed, 
\begin{align*}
&\phi\left( \sum_{i=1}^n \left< m, \lambda_i\right> x_i + \left< m, \lambda_{2n+1}\right>x_{2n+1}+ \left< m, \lambda_{2n+2}\right>x_{2n+2} + \left<m, \eta\right> x_\sigma \right) \\
&=\sum_{i=1}^n \left< m, \lambda_i\right> (x_i+x_{n+i}) + \left< m, \lambda_{2n+1}\right>x_{2n+1}+ \left< m, \lambda_{2n+2}\right>x_{2n+2} + \left<m, \eta\right> \sum_{i=1}^n c_ix_{n+i}\\
&=\sum_{i=1}^n\big(\left<m, \lambda_i\right>x_i +\left<m, \lambda_i+ c_i\eta\right>x_{n+i}\big) +\left< m, \lambda_{2n+1}\right>x_{2n+1}+ \left< m, \lambda_{2n+2}\right>x_{2n+2} \\
&=\sum_{i=1}^{2n+2} \left<m, \lambda_i\right>x_i \in \mathcal{J}. 
\end{align*}
For $x_jx_k\in \mathcal{J}_{\text{(1-1)}}$, we have 
\begin{align}\label{eq_phi_xixj}
\begin{split}
\phi(x_jx_k)=\begin{cases} (x_j+x_{n+j})(x_k+x_{n+k})& \text{if } 1\leq j,k \leq n;\\
(x_j+x_{n+j})x_{k} & \text{if } 1\leq j \leq n \text{ and } 2n+1\leq k \leq 2n+2;\\
x_j(x_{k}+x_{n+k}) & \text{if } 2n+1\leq j \leq 2n+2 \text{ and } 1\leq k \leq n;  \\
x_jx_k & \text{if } 2n+1\leq j\neq k \leq 2n+2.
\end{cases}
\end{split}
\end{align}
In any of above cases, $\phi(x_jx_k)\in \mathcal{J}$. Also, for $x_jx_\sigma\in \mathcal{J}_{\text{(1-1)}}$ with $1\leq j \leq n$, we have 
\begin{align*}
\phi(x_jx_\sigma)&=x_j\phi(x_\sigma)+x_{n+j}\phi(x_\sigma)\\
&=x_j\phi(x_\sigma)+x_{n+j}\sigma \left(\phi(x_\sigma)\right)\\
&=x_j\sum_{i=1}^nc_ix_{n+i} + x_{n+j}\sum_{i=1}^nc_ix_{i}=0,
\end{align*}
where the second equality follows from Lemma \ref{lem_image_invariant} and the last equality follows from \eqref{eq_ideal_polygon}. 
Hence, we have established the claim. 
\end{proof}

Since the automorphism $X_P$ induced from the reflection $\sigma$ is orientation preserving, $H^4(X_P)^\sigma\cong H^4(X_P/\sigma)$ is $1$-dimensional.  We also note that $H^\ast(X_{P/\sigma})$ is a Poincar\'e duality algebra and $H^4(X_{P/\sigma})$ is also $1$-dimensional as it is an orientable orbifold. 
Therefore, the map $\phi$ onto its image is an isomorphism by Proposition~\ref{prop_PD}, which completes the proof of Theorem \ref{thm_main} for Case (1-1).

\subsection{Proof of Theorem \ref{thm_main} for Case (1-2)}
We adhere to the notations discussed in the previous subsection. Then, $P$ is equipped with the conditions \ref{eq_ci} and~\ref{eq_2n+1_perp_eta}. The cohomology ring of $X_{P/\sigma}$ is given by 
\[
H^\ast(X_{P/\sigma}) = \QQ[x_1, \dots, x_n, x_{2n+1}, x_\sigma]/\mathcal{I}_{\text{(1-2)}}+\mathcal{J}_{\text{(1-2)}}
\]
where $\mathcal{I}_{\text{(1-2)}}$ is the ideal generated by
\[
\left\{ x_jx_k\mid  E_j \cap E_k = \emptyset  \text{ for } j,k \in \{1, \dots, n, 2n+1\} \right\}\cup \{x_jx_\sigma\mid 2\leq j\leq n\}
\]
and $\mathcal{J}_{\text{(1-2)}}$ is the ideal generated by 
\[
\left\{ \sum_{i=1}^n \left< m, \lambda_i\right> x_i + \left< m, \lambda_{2n+1}\right>x_{2n+1}+ \left<m, \eta\right> x_\sigma ~\Big|~  m\in M\right\}.
\]

Now, we consider the map $\phi \colon H^\ast(X_{P/\sigma}) \to H^\ast(X)$ defined in \eqref{eq_map_phi} except for the variable $x_{2n+2}$. With this setup, we have Lemmas \ref{lem_image_invariant} and \ref{lem_well_def} just by forgetting all monomials containing $x_{2n+2}$. The rest of the proof for Case (1-1) discussed in the previous subsection are valid for this case as well.  

\subsection{Proof of Theorem \ref{thm_main} for Case (1-3)}
In this case, $P$ satisfies \ref{eq_ci} only. 
The cohomology ring of $X_{P/\sigma}$ of this case is given by 
\[
H^\ast(X_{P/\sigma}) = \QQ[x_1, \dots, x_n,  x_\sigma]/\mathcal{I}_{\text{(1-3)}}+\mathcal{J}_{\text{(1-3)}}
\]
where $\mathcal{I}_{\text{(1-3)}}$ is the ideal generated by
\[
\left\{ x_jx_k\mid  E_j \cap E_k = \emptyset  \text{ for } j,k \in \{1, \dots, n\} \right\}\cup \{x_jx_\sigma\mid 2\leq j\leq n-1\}
\]
and $\mathcal{J}_{\text{(1-3)}}$ is the ideal generated by 
\[
\left\{ \sum_{i=1}^n \left< m, \lambda_i\right> x_i +  \left<m, \eta\right> x_\sigma ~\Big|~  m\in M\right\}.
\]
Here we also consider the map $\phi$ defined in \eqref{eq_map_phi} except for  variables $x_{2n+1}$ and $x_{2n+2}$. For the proof of Lemma \ref{lem_image_invariant}, one can take $\eta_2\in N$ such that $\lambda_i +\lambda_{n+i}=d_i\eta_2$ for some $d_i\in \ZZ$ for each $1\leq i \leq n$,  which plays a role of $\lambda_{2n+1}$, namely, $\{\eta_1=\eta, \eta_2\}$ form a basis of $N$. The proof of Lemma \ref{lem_well_def} and the rest of proof is straightforward.


\section{Symmetry by a dihedral group}\label{sec_two_ref}
In this section, we consider a polygon $P \subset M\otimes_\ZZ \RR$ with an action of the dihedral group $W\colonequals D_{2\ell}$ of \eqref{eq_dihedral_gp}. The fixed point set of $P$ by $s_i$ is given by
\[
E_{s_i} \colonequals \{\mathbf{x}\in P \mid s_i(x)=x\} = \{\mathbf{x}\in P \mid \left< \mathbf{x}, \eta_i \right>=0\}
\]
for some primitive vector $\eta_i\in N$ and for each $i=1, 2$. We define the fundamental region of $P$ with respect to the $W$-action by 
\begin{equation}\label{eq_fund_region_D2k}
P/W\colonequals\{\mathbf{x}\in P \mid \left< \mathbf{x}, \eta_1\right> \leq 0 \} \cap \{\mathbf{x}\in P \mid \left< \mathbf{x}, \eta_2\right> \leq 0 \}.
\end{equation}
Similarly to the discussion of Section \ref{sec_sing_ref}, we consider three cases of $P$ and symmetries by $W$ as described in Figure \ref{fig_D2k}. 

\begin{figure}
\begin{tikzpicture}
\node[opacity=0, regular polygon, regular polygon sides=12, draw, minimum size = 3cm](m) at (0,0) {};
\foreach \x in {1,...,12}{
\coordinate (\x) at (m.corner \x); 
}

\draw[fill=yellow, yellow] (0,0)--(m.side 1)--(1)--(12)--(m.side 11)--cycle;
\foreach \a in {1,2,11,12} {
\draw[fill] (\a) circle (1pt); 
}

\draw ($0.8*(2)+0.2*(3)$)--(2)--(1)--($0.8*(1)+0.2*(12)$);
\draw ($0.8*(11)+0.2*(10)$)--(11)--(12)--($0.8*(12)+0.2*(1)$);

\node[rotate=-30] at (m.side 12) {\scriptsize$\cdots$};
\node[rotate=30] at (m.side 2) {\scriptsize$\cdots$};

\node[rotate=90] at (m.side 10) {\scriptsize$\cdots$};

\draw[fill] (0,0) circle (1pt); 
\node[below right] at (0,0) {$O$};

\draw[blue] ($1.3*(m.side 1)$)--($0.5*(7)+0.5*(8)$);
\node at ($1.4*(m.side 1)$) {$s_1$};
\node at ($1.15*(m.side 1)$) {$\curvearrowleftright$};

\draw[blue] ($1.3*(m.side 11)$)--($0.5*(5)+0.5*(6)$);
\node at ($1.45*(m.side 11)$) {$s_2$};
\node[rotate=-60] at ($1.15*(m.side 11)$) {$\curvearrowleftright$};

\node at (0,-2.2) {Case (2-1)};

\begin{scope}[xshift=130]
\node[opacity=0, regular polygon, regular polygon sides=12, draw, minimum size = 3cm](m) at (0,0) {};
\foreach \x in {1,...,12}{
\coordinate (\x) at (m.corner \x); 
}

\draw[fill=yellow, yellow] (0,0)--(m.side 1)--(1)--(12)--(11)--cycle;
\foreach \a in {11,12,1,2} {
\draw[fill] (\a) circle (1pt); 
}
\node[rotate=-30] at (m.side 12) {\scriptsize$\cdots$};
\node[rotate=30] at (m.side 2) {\scriptsize$\cdots$};

\node[rotate=75] at ($0.95*(10)$) {\scriptsize$\cdots$};
\draw[fill] (0,0) circle (1pt); 
\node[below right] at (0,0) {$O$};

\draw ($0.8*(2)+0.2*(3)$)--(2)--(1)--($0.8*(1)+0.2*(12)$);
\draw (m.side 10)--(11)--(12)--($0.8*(12)+0.2*(1)$);

\draw[blue] ($1.3*(m.side 1)$)--($0.5*(7)+0.5*(8)$);
\node at ($1.4*(m.side 1)$) {$s_1$};
\node at ($1.15*(m.side 1)$) {$\curvearrowleftright$};

\draw[blue] ($1.3*(11)$)--($(5)$);
\node at ($1.45*(11)$) {$s_2$};
\node[rotate=-75] at ($1.15*(11)$) {$\curvearrowleftright$};

\node at (0,-2.2) {Case (2-2)};
\end{scope}

\begin{scope}[xshift=260]
\node[opacity=0, regular polygon, regular polygon sides=12, draw, minimum size = 3cm](m) at (0,0) {};
\foreach \x in {1,...,12}{
\coordinate (\x) at (m.corner \x); 
}

\draw[fill=yellow, yellow] (0,0)--(1)--(m.side 12)--(m.side 11)--(11)--cycle;
\foreach \a in {11,1} {
\draw[fill] (\a) circle (1pt); 
}
\node[rotate=-45] at ($0.95*(12)$) {\scriptsize$\cdots$};

\node[rotate=75] at ($0.95*(10)$) {\scriptsize$\cdots$};
\draw[fill] (0,0) circle (1pt); 
\node[below right] at (0,0) {$O$};

\draw (m.side 1)--(1)--(m.side 12);
\draw (m.side 11)--(11)--(m.side 10);

\draw[blue] ($1.3*(1)$)--($(7)$);
\node at ($1.4*(1)$) {$s_1$};
\node[rotate=-15] at ($1.15*(1)$) {$\curvearrowleftright$};

\draw[blue] ($1.3*(11)$)--($(5)$);
\node at ($1.45*(11)$) {$s_2$};
\node[rotate=-75] at ($1.15*(11)$) {$\curvearrowleftright$};

\node at (0,-2.2) {Case (2-3)};
\end{scope}

\end{tikzpicture}
\caption{Polygons with symmetries by $D_{2k}$ and their fundamental regions.}
\label{fig_D2k}
\end{figure}
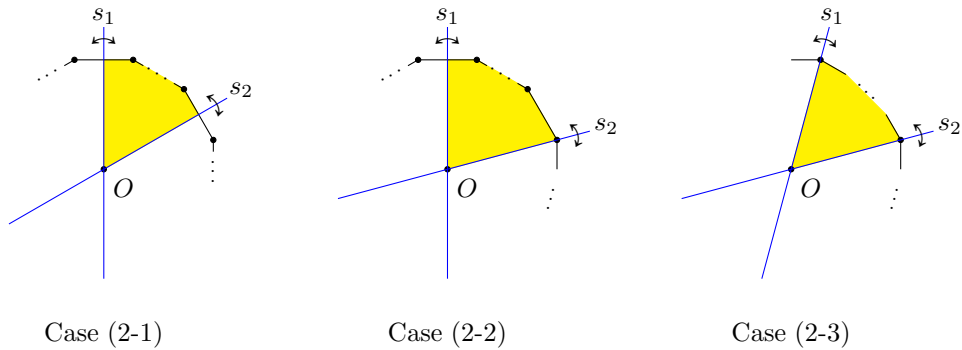

\subsection{Proof of Theorem \ref{thm_main} for Case (2-1)}\label{subsec_Case_2_1}
Notice that the polygon $P$ can be divided into $|W|$-copies of fundamental regions, each of which bijectively corresponds to elements of $W$. We fix the fundamental region $P/W$ of \eqref{eq_fund_region_D2k} as the one corresponding to the identity element $id \in W$. 

We label facets of $P/W$ as in Figure \ref{fig_facets_fund_region_C_2_1}-(i). Note that the facets $E_{1}$ and $E_{n}$ are fixed by $s_1$ and $s_2$, respectively. Each facet $E_{j}$ for $2\leq j \leq  n-1$ has the trivial stabilizer. Then, the labels on facets of $P/W$ naturally induces the labels on facets of $P$ by taking $W$-action on $\{E_{j}\mid 1\leq j \leq n\}$. To be more precise, let 
\[
 \prescript{s_i}{}{W}\colonequals \{u\in W \mid l(us_i) \geq l(u) \}
\]
be the minimal length left coset representatives of $s_i$ for $i=1,2$, where $l(u)$ denotes the length  of $u\in W$.  Then, the set $\mathcal{F}(P)$ of facets of $P$ is 
\begin{align}
\nonumber \mathcal{F}(P)& =\{u(E_j) \mid u\in W,~ 1\leq j \leq n \}\\
\label{eq_facet_decomp}&=\{u(E_1) \mid u\in { \prescript{s_1}{}{W}} \} \sqcup \left(\bigsqcup_{j=2}^{n-1}\{u(E_j) \mid u\in W\}\right)  \sqcup \{u(E_n) \mid u\in  { \prescript{s_2}{}{W}}\}
\end{align}
for $n\geq 2$. See Remark \ref{rmk_n_2} below for the case where $n=2$. The example of $12$-gon with $D_{6}$-symmetry is illustrated in See Figure \ref{fig_facets_fund_region_C_2_1}-(ii). 

\begin{remark}\label{rmk_n_2}
It is also possible that the middle term 
\begin{equation}\label{eq_middle_term_of_facets}
\bigsqcup_{j=2}^{n-1}\{u(E_j) \mid u\in W\}
\end{equation}
of \eqref{eq_facet_decomp} above is an empty set. An example of this type will be discussed in Section~\ref{sec_G_2}. In what follows, we pretend to have elements of \eqref{eq_middle_term_of_facets} for the full generality. For the case where \eqref{eq_middle_term_of_facets} is empty, omitting this term and assuming $n=2$ will lead us the proof. 
\end{remark}

\begin{figure}
\begin{tikzpicture}
\begin{scope}[rotate=7.5, scale=0.9]
\foreach \x in {1,2,...,6} {
\coordinate (\x) at (90-15*\x+15: 3); 
}

\draw[fill=yellow, yellow] ($1/2*(1)+1/2*(2)$)--(2)--(3)--(4)--(5)--($1/2*(5)+1/2*(6)$)--(0,0)--cycle;
\foreach \x in {1,2,...,6} {
\draw[fill] (\x) circle (1pt); 
}
\draw[fill] (0,0) circle (1pt);

\draw[blue] ($0.6*(1)+0.6*(2)$)--($-0.1*(1)-0.1*(2)$);
\draw[blue] ($0.6*(5)+0.6*(6)$)--($-0.1*(5)-0.1*(6)$);

\draw ($0.8*(1)+0.2*(105:3)$)--(1)--(2)--($0.8*(2)+0.2*(3)$);
\draw ($0.8*(3)+0.2*(2)$)--(3)--(4)--($0.8*(4)+0.2*(5)$);
\draw ($0.8*(5)+0.2*(4)$)--(5)--(6)--($0.8*(6)+0.2*(0:3)$);

\node[rotate=-15] at (75-7.5:3) {\scriptsize$\cdots$};
\node[rotate=-45] at (45-7.5:3) {\scriptsize$\cdots$};

\node[above] at ($0.5*(1)+0.5*(2)$) {\footnotesize$E_{1}$};
\node[rotate=-30] at ($0.55*(3)+0.55*(4)$) {\footnotesize$E_{j}$};
\node[rotate=-60] at ($0.55*(5)+0.55*(6)$) {\footnotesize$E_{n}$};

\node[left] at ($0.25*(1)+0.25*(2)$) {\footnotesize$E_{s_1}$};
\node[below right] at ($0.25*(5)+0.25*(6)$) {\footnotesize$E_{s_2}$};
\end{scope}

\begin{scope}[xshift=220, yshift=35]
\node[regular polygon, regular polygon sides=12, draw, minimum size = 3cm](m) at (0,0) {};
\foreach \x in {1,...,12}{
\coordinate (\x) at (m.corner \x); 
\draw[fill] (\x) circle (1pt); 
}

\draw[blue, dashed] ($0.7*(1)+0.7*(2)$)--($-0.5*(1)-0.5*(2)$);
\node at ($0.8*(1)+0.8*(2)$) {$s_1$};

\draw[blue, dashed] ($0.8*(11)+0.8*(12)$)--($-0.5*(11)-0.5*(12)$);
\node at ($0.9*(11)+0.9*(12)$) {$s_2$};

\draw[blue, dashed] (m.side 9)--(m.side 3);

\node[above] at (m.side 1) {\scriptsize$E_{1}$};
\node[rotate=0] at ($1.15*(m.side 12)$) {\scriptsize$E_{2}$};
\node[rotate=0] at ($1.2*(m.side 11)$) {\scriptsize$E_{3}$};

\node[rotate=0, right] at (m.side 10) {\scriptsize$s_2(E_{2})$};
\node[rotate=0] at ($1.3*(m.side 9)$) {\scriptsize$s_2(E_1)$};
\node[rotate=0] at ($0.8*(3)+0.4*(2)$) {\scriptsize$s_1(E_2)$};
\node[rotate=0] at ($1.3*(m.side 3)$) {\scriptsize$s_1(E_3)$};
\node[left] at ($(m.side 4)$) {\scriptsize$s_1s_2(E_2)$};
\node at ($1*(5)+0.4*(6)$) {\scriptsize$s_1s_2(E_1)$};
\node[below left] at (m.side 6) {\scriptsize$s_1s_2s_1(E_2)$};
\node[below] at (m.side 7) {\scriptsize$s_2s_1(E_{3})$};
\node[below right] at (m.side 8) {\scriptsize$s_2s_1(E_{2})$};

\end{scope}

\end{tikzpicture}
\caption{(i) Labeling on facets of $P/W$ of Case (2-1); (ii) Example: Labeling on facets of 12-gon with $D_{6}$-symmetry.}
\label{fig_facets_fund_region_C_2_1}
\end{figure}
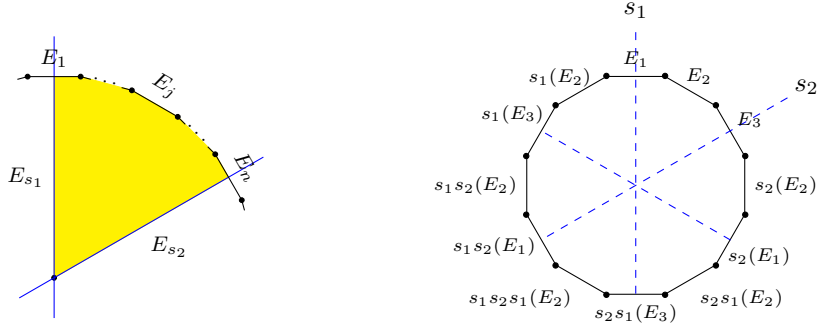

With this setup, we begin with rewriting the cohomology $H^\ast(X_P)$ of \eqref{eq_cohom_XP} as
\[
\QQ[x_E\mid E\in \mathcal{F}(P) ]/\mathcal{I}+\mathcal{J}
\]
where $\mathcal{I}=\left< x_{E}x_{E'} \mid E\cap E'=\emptyset\right>$ and 
$\mathcal{J}=\left< \sum_{E\in \mathcal{F}(P)}\left<m, \lambda(F)\right>x_E \mid m\in M \right>$. Then, $W$-module structure on $H^\ast(X_P)$ is given by $u(x_E)\colonequals x_{u(E)}$ for each $u\in W$. 

Recall from \eqref{eq_fund_region_D2k} that $\eta_1$ and $\eta_2$ are (outward) vectors to define facets $E_{s_1}$ and $E_{s_2}$, respectively. Let $\lambda(u(E_j))\in N$ be the  vector defining the facet $u(E_j)$ for~$u\in W$. Observe that for each $1\leq i\leq 2$ and $1\leq j\leq n$, we have 
\[
\lambda(s_i u(E_j))-\lambda(u(E_j))=c\cdot \eta_i
\]
for some $c\in \QQ$. If $u(E_j)$ is fixed by $s_i$, then $c=0$ obviously. Therefore, if $u=s_{i_r}s_{i_{r-1}}\cdots s_{i_1}$ is a reduced word expression of $u$, then 
\begin{align*}
\lambda(u(E_j))-\lambda(E_j)
&=\left(\lambda(s_{i_r}\cdots s_{i_1}(E_j))- \lambda(s_{i_{r-1}}\cdots s_{i_1}(E_j))\right) + \\
&\qquad \left(\lambda(s_{i_{r-1}}\cdots s_{i_1}(E_j))- \lambda(s_{i_{r-2}}\cdots s_{i_1}(E_j))\right)+ \cdots \\
& \qquad \cdots + \big(\lambda(s_{i_2}s_{i_1}(E_j))-\lambda(s_{i_1}(E_j))\big) +  \big(\lambda(s_{i_1}(E_j))-\lambda(E_j)\big)\\
&=\sum_{\substack{1\leq k \leq r; \\  i_k=1}}c_{u,j,k}\eta_1 + \sum_{\substack{1\leq k \leq r; \\  i_k=2}}d_{u,j, k}\eta_2
\end{align*}
for some $c_{u,j,k}, d_{u,j,k} \in \QQ$. Writing
$c_{u,j}\colonequals \sum_{\substack{1\leq k \leq r \\  i_k=1}}c_{u,j,k}$ and  
$d_{u,j}\colonequals \sum_{\substack{1\leq k \leq r \\  i_k=2}}d_{u,j, k}$
for simplicity,  we have 
\begin{equation}\label{eq_coeff}
\lambda(u(E_j))-\lambda(E_j)=c_{u,j}\eta_1+d_{u,j}\eta_2.
\end{equation}
We notice that 
\begin{equation}\label{eq_coef_c_d_zero}
c_{id,j}=c_{s_2,j}=d_{id, j}=d_{s_1,j}=0
\end{equation} 
for all $j$, which follows immediately from the definition. 

Now, we apply \eqref{eq_cohom_XP} to the toric surfaces $X_{P/W}$, which gives us
\begin{equation}\label{eq_cohom_XPG}
H^\ast(X_{P/W})=\QQ\left[ x_1, \dots, x_n, x_{s_1}, x_{s_2}\right]/\mathcal{I}_{\text{(2-1)}}+\mathcal{J}_{\text{(2-1)}},
\end{equation}
where $\mathcal{I}_{\text{(2-1)}}$ is the ideal generated by 
\[
\{ x_ix_j \mid E_i\cap E_j =\emptyset,~ 1\leq i, j \leq n\} \cup \{x_ix_{s_1} \mid 2\leq i \leq n\} \cup  \{x_ix_{s_2} \mid 1\leq i \leq n-1\} 
\]
and $\mathcal{J}_{\text{(2-1)}}$ is the ideal generated by 
\[
\left\{
\sum_{j=1}^n\left<m, \lambda(E_j)\right>x_j +\sum_{j=1}^2 \left<m, \eta_j\right>x_{s_j} ~\Big|~ m\in M
\right\}.
\]

We define a map 
\begin{equation}\label{eq_psi}
\psi \colon H^\ast(X_{P/W}) \to H^\ast(X_P)
\end{equation}
by 
\begin{align*}
x_j&\mapsto \begin{cases} 
\sum_{u\in { \prescript{s_1}{}{W}} }x_{u(E_1)}, &j=1;\\
\sum_{u\in W}x_{u(E_j)}, &j=2, \dots,  n-1;\\
\sum_{u\in { \prescript{s_2}{}{W}}}x_{u(E_n)}, &j=n,
\end{cases}\\
x_{s_1}&\mapsto \sum_{u\in  { \prescript{s_1}{}{W}}}  c_{u,1}x_{u(E_1)}+ \sum_{j=2}^{n-1} \sum_{u\in W}c_{u,j}x_{u(E_j)} + \sum_{u\in { \prescript{s_2}{}{W}} }c_{u,n}x_{u(E_n)} \\
x_{s_2}&\mapsto \sum_{u\in { \prescript{s_1}{}{W}}}  d_{u,1}x_{u(E_1)}+ \sum_{j=2}^{n-1} \sum_{u\in W}d_{u,j}x_{u(E_j)} + \sum_{u\in { \prescript{s_2}{}{W}}}d_{u,n}x_{u(E_n)}
\end{align*}
where $c_{u, j}$'s and $d_{u, j}$'s are defined in \eqref{eq_coeff}. 
Now, we prove the following two lemmas parallel to Lemmas \ref{lem_image_invariant} and \ref{lem_well_def}, respectively. 
\begin{lemma}\label{lem_image_invariant_case_2_1}
The image of $\psi$ is the invariant ring $H^\ast(X_P)^W$. 
\end{lemma}
\begin{proof}
First, it is obvious to see that each $\psi(x_j)$ for $j=1, \dots, n$ is $W$-invariant. We assert that $\psi(x_{s_1})$ and $\psi(x_{s_2})$ are $W$- invariant as well. Since $\eta_1$ and $\eta_2$ are linearly independent in $N\otimes_\ZZ \RR$, we take its dual basis $\{\eta_1^\ast, \eta_2^\ast\}\in M\otimes_\ZZ\RR$. 

Recall that the linear ideal $\mathcal{J}$ is generated by 
\begin{equation}\label{eq_elements_in_J_Case_2_1}
\sum_{u\in  { \prescript{s_1}{}{W}}}  \left<m, \lambda(u(E_1))\right> x_{u(E_1)}+ \sum_{j=2}^{n-1} \sum_{u\in W}\left<m, \lambda(u(E_j))\right> x_{u(E_j)} + \sum_{u\in  { \prescript{s_2}{}{W}} }\left<m, \lambda(u(E_n))\right> x_{u(E_n)}
\end{equation}
for $m\in M$. 
Applying $m=\eta_1^\ast$ and $\lambda(u(E_j))=c_{u,j}\eta_1 + d_{u,j}\eta_2+\lambda(E_j)$ of~\eqref{eq_coeff} to~\eqref{eq_elements_in_J_Case_2_1}, we get 
\begin{align}\label{eq_J_applying_eta_star}
\begin{split}
&\sum_{u\in  { \prescript{s_1}{}{W}}}  c_{u,1}x_{u(E_1)}+ \sum_{j=2}^{n-1} \sum_{u\in W}c_{u,j}x_{u(E_j)} + \sum_{u\in  { \prescript{s_2}{}{W}} }c_{u,n}x_{u(E_n)} + \\
&\sum_{u\in  { \prescript{s_1}{}{W}}}  \left<\eta_1^\ast, \lambda(E_1)\right> x_{u(E_1)}+ \sum_{j=2}^{n-1} \sum_{u\in W}\left<\eta_1^\ast, \lambda(E_j)\right> x_{u(E_j)} + \sum_{u\in  { \prescript{s_2}{}{W}} }\left<\eta_1^\ast, \lambda(E_n)\right> x_{u(E_n)},
\end{split}
\end{align}
which is 0 in $H^\ast(X_P)$ as it is an element of $\mathcal{J}$. Notice that the first three terms of~ \eqref{eq_J_applying_eta_star} are $\phi(x_{s_1})$ and the sum of the last three terms is $W$-invariant. Therefore, we conclude $\phi(x_{s_1})$ is $W$-invariant. Similarly, we apply $m=\eta_2^\ast$ together with  $\lambda(u(E_j))=c_{u,j}\eta_1 + d_{u,j}\eta_2+\lambda(E_j)$, one can see that $\phi(x_{s_2})$ is $W$-invariant as well. 

It remains to prove that  $H^\ast(X_P)^W \subset \im \psi$. The proof is similar to the proof of Lemma \ref{lem_image_invariant}. Indeed, the torus equivariant cohomology of $X_P$ is 
\[
H^\ast_T(X_P)=\QQ[x_E \mid E\in \mathcal{F}(P)]/\mathcal{I},
\] 
where $\mathcal{I}=\left< x_Ex_{E'}\mid E\cap E' =\emptyset\right>$. Observe that the invariant ring $H^\ast_T(X_P)^W$ is generated by 
\begin{equation}\label{eq_generator}
\left\{\sum_{u\in W} x_{u(E_i)} ~\Big|~ i=1, \dots, n\right\}= \big\{\psi(x_i) \mid i=1, \dots, n\big\}. 
\end{equation}
Indeed, if an element $y\in H^\ast_T(X_P)^W$ contains a monomial $cx_{u(E_i)}x_{v(E_j)}$ with nontrivial coefficient $c\in \QQ$, i.e., $u(E_i)$ and $v(E_j)$ intersect, then $y$ must contains 
\begin{equation}
c\cdot \sum_{w\in W}x_{wu(E_i)}x_{wv(E_j)}=c\cdot \sum_{u\in W}x_{u(E_i)} \cdot \sum_{v\in W}x_{v(E_j)},
\end{equation}
where the equality follows because of the ideal $\mathcal{I}$. 
Since the natural surjective homomorphism $H^\ast_T(X_P)\to H^\ast(X_P)$ is $W$-equivariant, the restriction $H^\ast_T(X_P)^W \to H^\ast(X_P)^W$ is also surjective. 
\end{proof}

\begin{lemma}\label{lem_well_def_case_2_1}
The map $\psi$ is well-defined. 
\end{lemma}
\begin{proof}
It suffices to show that $\ker \psi$ contains both $\mathcal{I}_{\text{(2-1)}}$ and $\mathcal{J}_{\text{(2-1)}}$ of \eqref{eq_cohom_XPG}. For the ideal~$\mathcal{J}_{\text{(2-1)}}$, 
\begin{align*}
&\psi\left(\sum_{j=1}^n\left<m, \lambda(E_j)\right>x_j +\left<m, \eta_1\right>x_{s_1}+\left<m, \eta_2\right>x_{s_2}\right)\\
&= \left<m, \lambda(E_1)\right> \sum_{u\in  { \prescript{s_1}{}{W}}} x_{u(E_1)}+  \sum_{j=2}^{n-1} \left<m, \lambda(E_j)\right>\sum_{u\in W}x_{u(E_j)} + \left<m, \lambda(E_n)\right> \sum_{u\in  { \prescript{s_2}{}{W}} }x_{u(E_n)}\\
&\qquad +\left<m, \eta_1\right>\left( \sum_{u\in  { \prescript{s_1}{}{W}}}  c_{u,1}x_{u(E_1)}+ \sum_{j=2}^{n-1} \sum_{u\in W}c_{u,j}x_{u(E_j)} + \sum_{u\in  { \prescript{s_2}{}{W}} }c_{u,n}x_{u(E_n)} \right) \\
&\qquad +\left<m, \eta_2\right>\left(\sum_{u\in  { \prescript{s_1}{}{W}}}  d_{u,1}x_{u(E_1)}+ \sum_{j=2}^{n-1} \sum_{u\in W}d_{u,j}x_{u(E_j)} + \sum_{u\in  { \prescript{s_2}{}{W}} }d_{u,n}x_{u(E_n)}\right)
\end{align*}
which agrees with \eqref{eq_elements_in_J_Case_2_1} by \eqref{eq_coeff}. 

Next, we consider the ideal $\mathcal{I}_{\text{(2-1)}}$. It is straightforward to see that $\psi(x_ix_j)\in \mathcal{I}$ for~$x_ix_j\in \mathcal{I}_{\text{(2-1)}}$ by a similar computation of \eqref{eq_phi_xixj}. For $x_ix_{s_1}\in \mathcal{J}_{\text{(2-1)}}$ with $2\leq i \leq n$, 
\begin{equation}\label{eq_well_def_SR_Case_2_1}
\psi(x_ix_{s_1})=\left(\sum_{v\in W}x_{v(E_i)}\right)\cdot \psi(x_{s_1})=\sum_{v\in W}\left(x_{v(E_i)}\cdot v\left(\psi(x_{s_1})\right)\right),
\end{equation}
where the second equality follows by Lemma \ref{lem_image_invariant_case_2_1}. 
The right-most term of \eqref{eq_well_def_SR_Case_2_1} for~$v=id$ is 
\begin{equation}\label{eq_v_id}
x_{E_i}\psi(x_{s_1})= x_{E_i}\cdot\left( \sum_{u\in  { \prescript{s_1}{}{W}}}  c_{u,1}x_{u(E_1)}+ \sum_{j=2}^{n-1} \sum_{u\in W}c_{u,j}x_{u(E_j)} + \sum_{u\in { \prescript{s_2}{}{W}} }c_{u,n}x_{u(E_n)}\right).
\end{equation}
For $i=1, \dots, n-1$, we have $c_{id, i-1}=c_{id, i}=c_{id, i+1}=0$ by \eqref{eq_coef_c_d_zero}. For $i=n$, we also have $c_{id,n-1}=c_{id, n}=c_{s_2, n-1}=0$ by \eqref{eq_coef_c_d_zero} again. Therefore, we conclude~\eqref{eq_v_id} is~$0$ in $H^\ast(X_P)$ for all $i=2, \dots n$. The computation for arbitrary $v\in W$ follows by Lemma \ref{lem_image_invariant_case_2_1}. Hence, we have established that \eqref{eq_well_def_SR_Case_2_1} is $0$ in $H^\ast(X_P)$. 
\end{proof}

The rest of the proof is essentially same as the last part of Subsection \ref{subsec_Case_1_1}. To be more precise, both $H^\ast(X_{P/W})$ and $H^\ast(X_P)^W$ are of dimension 1 as $W$-action on $X_P$ is orientation preserving. So, Proposition~\ref{prop_PD} completes the proof because $H^\ast(X_{P/W})$ is a Poincar\'e duality algebra.

\begin{remark}
Toric varieties associated with Type $A_2, B_2, C_2$ and $D_2$ root systems together with the actions of Weyl groups are in the category of Case (2-1). We refer to \cite{HMSS}. In Section \ref{sec_G_2}, we consider the toric variety associated with the root system of type $G_2$. 
\end{remark}

\subsection{Proof of Theorem \ref{thm_main} for Case (2-2)}
The set $\mathcal{F}(P)$ of facets of $P$ with reflections $s_1$ and $s_2$ given in Case (2-2) of Figure \ref{fig_D2k} is 
\[
\mathcal{F}(P)=
\{u(E_1) \mid u\in  { \prescript{s_1}{}{W}}\} \sqcup \left(\bigsqcup_{j=2}^{n-1}\{u(E_j) \mid u\in W\}\right),
\]
see Figure \ref{fig_labeling_Case_2_2_and_2_3}. 
\begin{figure}
\begin{tikzpicture}[, scale=0.8]
\begin{scope}[rotate=7.5]
\foreach \x in {1,2,...,6} {
\coordinate (\x) at (90-15*\x+15: 3); 
}
\draw[fill] (0,0) circle (1pt); 

\draw[fill=yellow, yellow] ($1/2*(1)+1/2*(2)$)--(2)--(3)--(4)--(5)--(6)--(0,0)--cycle;
\foreach \x in {1,2,...,6} {
\draw[fill] (\x) circle (1pt); 
}
\draw[blue] ($0.6*(1)+0.6*(2)$)--($-0.1*(1)-0.1*(2)$);
\draw[blue] ($1.2*(6)$)--($-0.2*(6)$);

\draw ($0.8*(1)+0.2*(105:3)$)--(1)--(2)--($0.8*(2)+0.2*(3)$);
\draw ($0.8*(3)+0.2*(2)$)--(3)--(4)--($0.8*(4)+0.2*(5)$);
\draw ($0.8*(5)+0.2*(4)$)--(5)--(6)--($0.8*(6)+0.2*(0:3)$);

\node[rotate=-15] at (75-7.5:3) {\scriptsize$\cdots$};
\node[rotate=-45] at (45-7.5:3) {\scriptsize$\cdots$};

\node[above] at ($0.5*(1)+0.5*(2)$) {\scriptsize$E_{1}$};
\node[rotate=-30] at ($0.55*(3)+0.55*(4)$) {\scriptsize$E_{j}$};
\node[rotate=-60] at ($0.55*(5)+0.55*(6)$) {\scriptsize$E_{n-1}$};

\node[left] at ($0.25*(1)+0.25*(2)$) {\scriptsize$E_{s_1}$};
\node[below right] at ($0.5*(6)$) {\scriptsize$E_{s_2}$};
\end{scope}

\begin{scope}[xshift=200]
\foreach \x in {1,2,...,6} {
\coordinate (\x) at (90-15*\x+15: 3); 
}
\draw[fill] (0,0) circle (1pt); 

\draw[fill=yellow, yellow] (1)--(2)--(3)--(4)--(0,0)--cycle;

\foreach \x in {1,2,...,4} {
\draw[fill] (\x) circle (1pt); 
}

\draw[blue] ($1.2*(1)$)--($-0.2*(1)$);
\draw[blue] ($1.2*(4)$)--($-0.2*(4)$);

\draw ($0.8*(1)+0.2*(105:3)$)--(1)--(2)--($0.8*(2)+0.2*(3)$);
\draw ($0.8*(3)+0.2*(2)$)--(3)--(4)--($0.8*(4)+0.2*(5)$);

\node[rotate=-15] at (75-7.5:3) {\scriptsize$\cdots$};

\node[above] at ($0.5*(1)+0.5*(2)$) {\scriptsize$E_{2}$};
\node[rotate=-30] at ($0.55*(3)+0.55*(4)$) {\scriptsize$E_{n-1}$};

\node[left] at ($0.5*(1)$) {\scriptsize$E_{s_1}$};
\node[below right] at ($0.5*(4)$) {\scriptsize$E_{s_2}$};
\end{scope}
\end{tikzpicture}
\caption{Labelings of facets of $P/W$ for Cases (2-2) and (2-3). }
\label{fig_labeling_Case_2_2_and_2_3}
\end{figure}
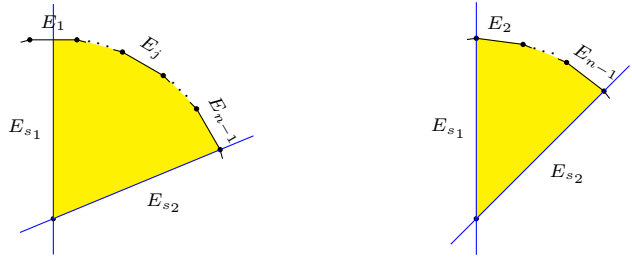
Accordingly, \eqref{eq_cohom_XP} gives us the cohomology of $X_{P/W}$ as follows:
\begin{equation*}\label{eq_cohom_Case_2_2}
H^\ast(X_{P/W})=
\QQ[x_1, \dots, x_{n-1}, x_{s_1}, x_{s_2}]/\mathcal{I}_{\text{(2-2)}}+\mathcal{J}_{\text{(2-2)}},
\end{equation*}
where 
$\mathcal{I}_{\text{(2-2)}}$ is the ideal generated by 
\[
\{ x_ix_j \mid E_i\cap E_j =\emptyset,~ 1\leq i, j \leq n-1\} \cup \{x_ix_{s_1} \mid 2\leq i \leq n-1\} \cup  \{x_ix_{s_2} \mid 1\leq i \leq n-2\} 
\]
and $\mathcal{J}_{\text{(2-2)}}$ is the ideal generated by 
\[
\left\{
\sum_{j=1}^{n-1}\left<m, \lambda(E_j)\right>x_j +\sum_{j=1}^2 \left<m, \eta_j\right>x_{s_j} ~\Big|~ m\in M
\right\}.
\]
When $n=3$ in particular, the first part $\{ x_ix_j \mid E_i\cap E_j =\emptyset,~ 1\leq i, j \leq n-1\}$ of generators of $\mathcal{I}_{\text{(2-2)}}$ is empty. 

Now, it is straightforward to follow the proofs of  Lemmas \ref{lem_image_invariant_case_2_1} and \ref{lem_well_def_case_2_1} together with the map $\psi \colon H^\ast(X_{P/W}) \to H^\ast(X_P)$ defined in \eqref{eq_psi} with omitting  $\psi(x_n)$. Also, the rest of the proof is the same as the one given in the previous subsection.

%

\subsection{Proof of Theorem \ref{thm_main} for Case (2-3)}
In this case, facets of $P$ is 
\[
\mathcal{F}(P)=
\bigsqcup_{j=2}^{n-1}\{u(E_j) \mid u\in W\}. 
\]
In particular when $n=3$, $\mathcal{F}(P)=\{ u(E_2)\mid u\in W\}$. 
Therefore, the cohomology ring of~$X_{P/W}$ is 
\begin{equation*}\label{eq_cohom_Case_2_3}
H^\ast(X_{P/W})=
\QQ[x_2, \dots, x_{n-1}, x_{s_1}, x_{s_2}]/\mathcal{I}_{\text{(2-3)}}+\mathcal{J}_{\text{(2-3)}},
\end{equation*}
where 
$\mathcal{I}_{\text{(2-3)}}$ is the ideal generated by 
\[
\begin{cases}
\begin{array}{l}
 \{ x_ix_j \mid E_i\cap E_j =\emptyset,~ 2\leq i, j \leq n-1\} \\
 \quad \cup \{x_ix_{s_1} \mid 3\leq i \leq n-1\} \cup  \{x_ix_{s_2} \mid 2\leq i \leq n-2\} 
\end{array}
& \text{ if } n\geq 5;\\
~\{x_ix_{s_1} \mid 3\leq i \leq n-1\} \cup  \{x_ix_{s_2} \mid 2\leq i \leq n-2\}  
& \text{ if } n=4;\\
~~x_2x_{s_1}x_{s_2} &\text{ if } n=3,
\end{cases}
\]
and $\mathcal{J}_{\text{(2-3)}}$ is the ideal generated by 
\[
\left\{
\sum_{j=2}^{n-1}\left<m, \lambda(E_j)\right>x_j +\sum_{j=1}^2 \left<m, \eta_j\right>x_{s_j} ~\Big|~ m\in M
\right\}.
\]
Lemma \ref{lem_image_invariant_case_2_1} for this case follows by a similar computation with minor modification. For the proof of Lemma \ref{lem_well_def_case_2_1}, it is necessary to consider the case where $n=3$, which realizes $P/W$ as a triangle. In this case, 
\begin{align*}
\psi(x_2x_{s_1}x_{s_2})=\left(\sum_{u\in W}x_{u(E_2)}\right)\cdot \psi(x_{s_1}) \cdot \psi(x_{s_2})=\sum_{u\in W}\left( x_{u(E_2)} \cdot \psi(x_{s_1}) \cdot \psi(x_{s_2}) \right).
\end{align*}
For $id\in W$, 
\begin{align}\label{eq_SR_u_is_id}
\begin{split}
x_{E_2} \cdot \psi(x_{s_1}) \cdot \psi(x_{s_2})&= x_{E_2}\cdot \left( \sum_{u\in W}c_{u,2}x_{u(E_2)}\right) \cdot  \left(\sum_{u\in W}d_{u,2}x_{u(E_2)}\right)
 \\
&=\left(x_{E_2}\cdot c_{s_1,2}x_{s_1(E_2)}\right)\cdot \left(\sum_{u\in W}d_{u,2}x_{u(E_2)}\right)=0
\end{split}
\end{align}
where the second equality follows because $c_{id,2}=c_{s_2,2}=0$ (see \eqref{eq_coef_c_d_zero}) and $E_2$ intersects $s_i(E_2)$ for $i=1, 2$ only. The last equality follows because $d_{id,2}=d_{s_1, 2}=0$ and $s_1(E_2)$ does not intersect $s_2(E_2)$. 

Similarly for general $u\in W$, using the result of Lemma \ref{lem_image_invariant_case_2_1}, one can show that 
\[
 x_{u(E_2)} \cdot \psi(x_{s_1}) \cdot \psi(x_{s_2})=
x_{u(E_2)} \cdot u\left(\psi(x_{s_1})\right) \cdot u\left(\psi(x_{s_2})\right)=0
\]
where the first equality follows using the result of Lemma \ref{lem_image_invariant_case_2_1} and the second equality follows by the same reason for \eqref{eq_SR_u_is_id}.

\section{Example: toric variety associated with $G_2$-root system}\label{sec_G_2}
Let $V$ be the subspace of $\RR^3$ orthogonal to $(1,1,1)\in \RR$ and $\Phi_{G_2}$ the root system of type $G_2$ in $V$. We take simple roots $\{\alpha_1\colonequals (1,-1,0), \alpha_2\colonequals (-1,2,-1)\}$ and fundamental  coweights $\{\omega_1, \omega_2\}\subset V^\ast$, namely the dual basis of $\{\alpha_1, \alpha_2\}\subset V$. Following the notations of Section \ref{sec_toric_surface}, we take $M$ and $N$ to be the root lattice and coweight lattice, respectively. 
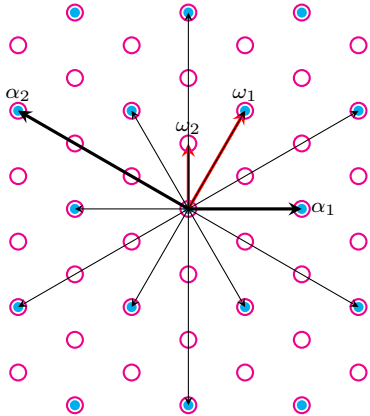
\begin{figure}
\begin{tikzpicture}[scale=1]

\node[opacity=0, regular polygon, regular polygon sides=6, draw, minimum size = 3cm](m) at (0,0) {};
\foreach \x in {1,...,6}{
\coordinate (\x) at (m.corner \x); 
}

\foreach \x in {0,1,2,3}{
\draw[cyan, thick, fill] ($(2)+(3)+\x*(6)$) circle (1.5pt); 
\draw[cyan, thick,fill] ($(3)+(4)+\x*(6)$) circle (1.5pt); 
}

\foreach \x in {-1,0,1}{
\draw[cyan, thick,fill] ($(0,0)+\x*(6)$) circle (1.5pt); 
\draw[cyan, thick,fill] ($(1)+(2)+\x*(6)$) circle (1.5pt); 
\draw[cyan, thick,fill] ($(4)+(5)+\x*(6)$) circle (1.5pt); 
}

\node[right] at (6) {\footnotesize$\alpha_1$};
\node[above] at ($(2)+(3)$) {\footnotesize$\alpha_2$};

\coordinate (w1) at ($2*(6)+(2)+(3)$);
\coordinate (w2) at ($(6)+2/3*(2)+2/3*(3)$);

\draw[red, very thick, -stealth] (0,0)--(w1); 
\draw[red, very thick, -stealth] (0,0)--(w2); 

\node[above] at (w1) {\footnotesize$\omega_1$};
\node[above] at (w2) {\footnotesize$\omega_2$};

\foreach \y in {-3,-2,-1,0,1,2,3}{
\draw[magenta, thick] ($\y*(w2)$) circle (3pt); 
\draw[magenta, thick] ($2*(w1)-3*(w2)+\y*(w2)$) circle (3pt); 
\draw[magenta, thick] ($-2*(w1)+3*(w2)+\y*(w2)$) circle (3pt); 
}

\foreach \y in {1,0,-1,-2,-3,-4}{
\draw[magenta, thick] ($-1*(w1)+3*(w2)+\y*(w2)$) circle (3pt); 
\draw[magenta, thick] ($1*(w1) + \y*(w2)$) circle (3pt); 
\draw[magenta, thick] ($3*(w1) + \y*(w2) -3*(w2)$) circle (3pt); 
\draw[magenta, thick] ($-3*(w1) + \y*(w2) +6*(w2)$) circle (3pt); 
}

\draw[very thick, -stealth] (0,0)--(6);
\draw[very thick, -stealth] (0,0)--($(2)+(3)$);
\draw[stealth-stealth] (1)--(4);
\draw[stealth-stealth] (2)--(5);
\draw[stealth-] (3)--(0,0);

\draw[stealth-stealth] ($(1)+(2)$)--($(4)+(5)$);
\draw[stealth-stealth] ($(2)+(3)$)--($(5)+(6)$);
\draw[stealth-stealth] ($(3)+(4)$)--($(6)+(1)$);

\end{tikzpicture}
\caption{The root lattice \textcolor{cyan}{$\bullet$} and coweight lattice \textcolor{magenta}{$\pmb\ovoid$} of type $G_2$.}
\label{fig_G_2_root_system}
\end{figure}
See Figure \ref{fig_G_2_root_system} for the description of $M$ and $N$, where we draw two lattices in the same plane using the coordinate presentations with respect to the standard basis of $\RR^3$ and its dual basis.

The Weyl group $W\colonequals W_{G_2}$ of type $G_2$ is generated by reflections $s_1$ and $s_2$ with respect to hyperplanes determined by $\alpha_1$ and $\alpha_2$ respectively. In this case, $W$ is isomorphic to the Dihedral group $D_{12}$. Now we consider the polytope 
\begin{align}\label{eq_fund_reg_G_2}
P \colonequals 
\bigcap_{u\in { \prescript{s_1}{}{W}}} \left\{ \mathbf{x}\in  V  \mid \left< \mathbf{x}, u\omega_2\right>+ a_u \leq 0\right\} \cap \bigcap_{v\in { \prescript{s_2}{}{W}}} \left\{ \mathbf{x}\in  V  \mid \left< \mathbf{x}, v\omega_1\right>+ b_v \leq 0\right\}
\end{align}
for some $a_u, b_v<0$, which can be thought as a \emph{weight polytope} of type $G_2$. See for instance \cite[Section 4]{Pos} for a general definition of a weight polytope. 

The fundamental coweight $\omega_1$ determines the edge $E_2$ and $\omega_2$ determines $E_1$. Accordingly, the edges~$u(E_1)$ and $u(E_2)$ are determined by $u\omega_2$ and $u\omega_1$, respectively, for $u\in W$. To be more precise, the dual representation of the Weyl group $W$ on the coweight space induced from the usual $W$-action on the root space is given by 
\begin{equation}\label{eq_W_rep_coweight_sp}
s_1 \colon \begin{array}{l} \omega_1 \mapsto -\omega_1+3\omega_2 \\ \omega_2 \mapsto \omega_2\end{array} \quad \text{ and } \quad 
s_2 \colon \begin{array}{l} \omega_1 \mapsto \omega_1 \\ \omega_2 \mapsto \omega_1-\omega_2.\end{array}
\end{equation}
Following \eqref{eq_fund_reg_G_2} together with \eqref{eq_W_rep_coweight_sp}, we illustrate complete information of vectors defining facets in Figure~ \ref{fig_G_2_rootsystem_weight_poly}.

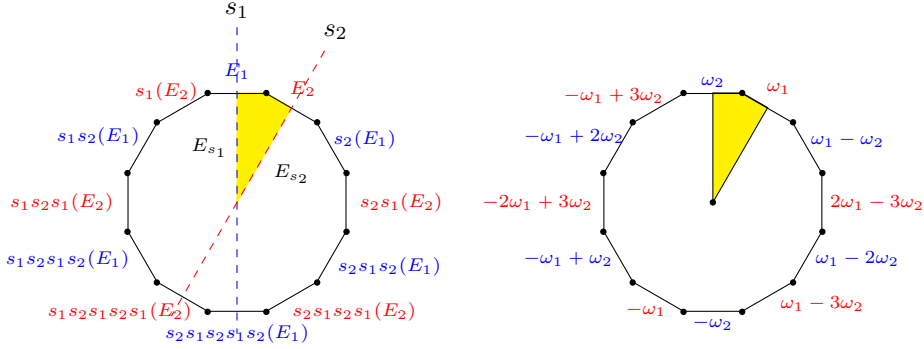
\begin{figure}
\begin{tikzpicture}

\node[opacity=0, regular polygon, regular polygon sides=12, draw, minimum size = 3cm](m) at (0,0) {};
\foreach \x in {1,...,12}{
\coordinate (\x) at (m.corner \x); 
}

\draw[fill=yellow, yellow] (0,0)--($0.5*(1)+0.5*(2)$)--(1)--($0.5*(1)+0.5*(12)$)--cycle;

\node[regular polygon, regular polygon sides=12, draw, minimum size = 3cm](m) at (0,0) {};

\foreach \x in {1,...,12}{
\coordinate (\x) at (m.corner \x); 
\draw[fill] (\x) circle (1pt); 
}

\draw[blue, dashed] ($0.8*(1)+0.8*(2)$)--($-0.6*(1)-0.6*(2)$);
\draw[red, dashed] ($0.8*(1)+0.8*(12)$)--($-0.6*(1)-0.6*(12)$);

\node[above] at ($0.8*(1)+0.8*(2)$) {$s_1$};
\node at ($0.9*(1)+0.9*(12)$) {$s_2$};

\node[red] at ($0.6*(1)+0.6*(12)$) {\scriptsize$E_2$};
\node[red]  at ($0.5*(2)+0.7*(3)$) {\scriptsize$s_1(E_2)$};
\node[red]  at ($0.75*(10)+0.75*(11)$) {\scriptsize$s_2s_1(E_2)$};
\node[red]  at ($0.8*(4)+0.8*(5)$) {\scriptsize$s_1s_2s_1(E_2)$};
\node[red]  at ($-0.1*(8)+1.5*(9)$) {\scriptsize$s_2s_1s_2s_1(E_2)$};
\node[red]  at ($1.5*(6)-0.1*(7)$) {\scriptsize$s_1s_2s_1s_2s_1(E_2)$};

\node[blue]  at ($0.6*(1)+0.6*(2)$) {\scriptsize$E_1$};
\node[blue]  at ($0.8*(11)+0.5*(12)$) {\scriptsize$s_2(E_1)$};
\node[blue]  at ($0.5*(3)+0.9*(4)$) {\scriptsize$s_1s_2(E_1)$};
\node[blue]  at ($0.4*(9)+1.1*(10)$) {\scriptsize$s_2s_1s_2(E_1)$};
\node[blue]  at ($1.4*(5)+0.2*(6)$) {\scriptsize$s_1s_2s_1s_2(E_1)$};
\node[blue]  at ($0.6*(7)+0.6*(8)$) {\scriptsize$s_2s_1s_2s_1s_2(E_1)$};

\node[left] at ($1/4*(1)+1/4*(2)$) {\scriptsize$E_{s_1}$};
\node[below right] at ($1/4*(1)+1/4*(12)$) {\scriptsize$E_{s_2}$};

\begin{scope}[xshift=180]

\node[opacity=0, regular polygon, regular polygon sides=12, draw, minimum size = 3cm](m) at (0,0) {};
\foreach \x in {1,...,12}{
\coordinate (\x) at (m.corner \x); 
}

\draw[fill=yellow] (0,0)--($0.5*(1)+0.5*(2)$)--(1)--($0.5*(1)+0.5*(12)$)--cycle;

\node[regular polygon, regular polygon sides=12, draw, minimum size = 3cm](m) at (0,0) {};

\foreach \x in {1,...,12}{
\coordinate (\x) at (m.corner \x); 
\draw[fill] (\x) circle (1pt); 
}
\draw[fill] (0,0) circle (1pt); 

\node[red] at ($0.6*(1)+0.6*(12)$) {\scriptsize$\omega_1$};
\node[red]  at ($0.1*(2)+1.2*(3)$) {\scriptsize$-\omega_1+3\omega_2$};
\node[red]  at ($0.75*(10)+0.75*(11)$) {\scriptsize$2\omega_1-3\omega_2$};
\node[red]  at ($0.8*(4)+0.8*(5)$) {\scriptsize$-2\omega_1+3\omega_2$};
\node[red]  at ($-0.1*(8)+1.4*(9)$) {\scriptsize$\omega_1-3\omega_2$};
\node[red]  at ($0.65*(6)+0.5*(7)$) {\scriptsize$-\omega_1$};

\node[blue]  at ($0.55*(1)+0.55*(2)$) {\scriptsize$\omega_2$};
\node[blue]  at ($0.9*(11)+0.45*(12)$) {\scriptsize$\omega_1-\omega_2$};
\node[blue]  at ($0.5*(3)+0.9*(4)$) {\scriptsize$-\omega_1+2\omega_2$};
\node[blue]  at ($0.3*(9)+1.1*(10)$) {\scriptsize$\omega_1-2\omega_2$};
\node[blue]  at ($1.1*(5)+0.3*(6)$) {\scriptsize$-\omega_1+\omega_2$};
\node[blue]  at ($0.55*(7)+0.55*(8)$) {\scriptsize$-\omega_2$};
\end{scope}
\end{tikzpicture}
\caption{(LHS) The weight polytope of $G_2$ root system and the fundamental region for $W$-action; (RHS) Vectors defining facets.}
\label{fig_G_2_rootsystem_weight_poly}
\end{figure}

The the fundamental region $P/W$ is colored in Figure \ref{fig_G_2_rootsystem_weight_poly} which is in the category of Case (2-1) of Section \ref{sec_two_ref}. We refer to Remark \ref{rmk_n_2}. Two edges $E_{s_1}$ and $E_{s_2}$  are determined by (minus of ) coroots 
\[
\{-\alpha_1^\vee=-2\omega_1+3\omega_2, -\alpha_2^\vee=\omega_1-2\omega_2\} \subset V^\ast
\] 
and they will play roles of $\eta_1, \eta_2$ in Section~\ref{sec_two_ref}. 
\begin{remark}
Blume \cite{Blu} introduced toric orbifolds associated to Cartan matrices for classical types. The toric variety $X_{P/W}$ corresponding to the fundamental region $P/W$ of \eqref{eq_fund_reg_G_2} can also be understood as a toric orbifold associated to the Cartan matrix of type $G_2$.
\end{remark}

Below, we calculate integers $\{c_{u,1}, c_{u,2}, d_{u,1}, d_{u,2}\mid u\in W\}$ of \eqref{eq_coeff}. Recall from~\eqref{eq_coef_c_d_zero} that  
\[
c_{id,1}=c_{id,2}=c_{s_2,1}=d_{id,1}=d_{id,2}=d_{s_1,2}=0.
\]  
We calculate the rest of integers $c_{u, i}$ and $d_{u, i}$ for $i=1, 2$:
\[
\begin{array}{rclcl}
\lambda(s_2(E_1))-\lambda(E_1)&=&\omega_1-2\omega_2&=&-\alpha_2^\vee,\\
\lambda(s_1s_2(E_1))-\lambda(E_1)&=&-\omega_1+\omega_2&=&-\alpha_1^\vee-\alpha_2^\vee,\\
\lambda(s_2s_1s_2(E_1))-\lambda(E_1)&=& \omega_1-3\omega_2&=&-\alpha_1^\vee-3\alpha_2^\vee,\\
\lambda(s_1s_2s_1s_2(E_1))-\lambda(E_1)&=& -\omega_1&=&-2\alpha_1^\vee-3\alpha_2^\vee,\\
\lambda(s_2s_1s_2s_1s_2(E_1))-\lambda(E_1)&=& -2\omega_2&=&-2\alpha_1^\vee-4\alpha_2^\vee,\\
\lambda(s_1(E_2))-\lambda(E_2)&=& -2\omega_1+3\omega_2&=&-\alpha_1^\vee,\\
\lambda(s_2s_1(E_2))-\lambda(E_2)&=& \omega_1-3 \omega_2&=&-\alpha_1^\vee-3\alpha_2^\vee,\\
\lambda(s_1s_2s_1(E_2))-\lambda(E_2)&=& -3\omega_1+3 \omega_2&=&-3\alpha_1^\vee-3\alpha_2^\vee,\\
\lambda(s_2s_1s_2s_1(E_2))-\lambda(E_2)&=& -3 \omega_2&=&-3\alpha_1^\vee-6\alpha_2^\vee,\\
\lambda(s_1s_2s_1s_2s_1(E_2))-\lambda(E_2)&=& -2 \omega_2&=&-4\alpha_1^\vee-6\alpha_2^\vee.
\end{array}
\]
The result of the above computation is summarized in Table \ref{Tab_coef}, and it gives 
the map $\psi\colon H^\ast(X_{P/W}) \to H^\ast(X_P)$ of \eqref{eq_psi} is defined by 
\begin{align*}
x_{E_1} &\mapsto x_{E_1}+ x_{s_2(E_1)}+x_{s_1s_2(E_1)}+x_{s_2s_1s_2(E_1)}+x_{s_1s_2s_1s_2(E_1)}+x_{s_2s_1s_2s_1s_2(E_1)};\\
x_{E_2} &\mapsto x_{E_1}+ x_{s_1(E_2)}+x_{s_2s_1(E_2)}+x_{s_1s_2s_1(E_2)}+x_{s_2s_1s_2s_1(E_2)}+x_{s_1s_2s_1s_2s_1(E_1)};\\
x_{E_{s_1}}&\mapsto \sum_{u\in  { \prescript{s_1}{}{W}}} c_{u,1}x_{u(E_1)} + \sum_{u\in  { \prescript{s_2}{}{W}}} c_{u,2}x_{u(E_2)};\\
x_{E_{s_2}}&\mapsto \sum_{u\in  { \prescript{s_1}{}{W}}} d_{u,1}x_{u(E_1)} + \sum_{u\in  { \prescript{s_2}{}{W}}} d_{u,2}x_{u(E_2)}.
\end{align*}
Now, one can show that $\psi$ induces an isomorphism $H^\ast(X_{P/W})\to H^\ast(X_P)^W$. 

\begin{table}[]
\def\arraystretch{1.2}
\begin{tabular}{C|CCCCCC}\hline 
u\in  { \prescript{s_1}{}{W}} & id & s_2 & s_1s_2 & s_2s_1s_2 & s_1s_2s_1s_2 & s_2s_1s_2s_1s_2 \\ \hline 
c_{u,1}        & 0  & 0   & 1      & 1         & 2            & 2               \\ \hline
d_{u,1}        & 0  & 1   & 1      & 3         & 3            & 4               \\ \hline \hline
u\in  { \prescript{s_2}{}{W}} & id & s_1 & s_2s_1 & s_1s_2s_1 & s_2s_1s_2s_1 & s_1s_2s_1s_2s_1 \\ \hline 
c_{u,2}        & 0  & 1   & 1      & 3         & 3            & 4               \\ \hline 
d_{u,2}        & 0  & 0   & 3      & 3         & 6            & 6              \\ \hline
\end{tabular}
\caption{List of coefficients for $\psi(x_{E_{s_1}})$ and $\psi(x_{E_{s_2}})$.}
\label{Tab_coef}
\end{table}

%
%

\end{document}